\documentclass[11pt,reqno]{amsart}

\usepackage[a4paper,margin=1in]{geometry}
\usepackage{amsmath,amssymb,amsthm,mathtools}
\usepackage{bm}
\usepackage{enumitem}
\usepackage{microtype}
\usepackage{booktabs}
\usepackage{graphicx}
\usepackage{cite}
\usepackage{hyperref}
\hypersetup{colorlinks=true,linkcolor=blue,citecolor=blue,urlcolor=blue}

\numberwithin{equation}{section}

\newtheorem{theorem}{Theorem}[section]
\newtheorem{proposition}[theorem]{Proposition}
\newtheorem{lemma}[theorem]{Lemma}
\newtheorem{corollary}[theorem]{Corollary}
\newtheorem{remark}[theorem]{Remark}

\newtheorem{example}[theorem]{Example}

\newcommand{\R}{\mathbb{R}}
\newcommand{\N}{\mathbb{N}}
\newcommand{\GOE}{\operatorname{GOE}}
\newcommand{\diag}{\operatorname{diag}}

\newcommand{\one}{\mathbf{1}}
\newcommand{\av}[1]{\left\langle #1\right\rangle}

\title[GOE in nonlinear ribosome flow models]{Gain of Entrainment in Ribosome Flow Models with Nonlinear Steric Hindrance: Global Rigidity and Dimension-Induced Gain}

\author{Hui Zhou}
\address{School of Mathematics and Statistics, Hefei Normal University, Hefei, Anhui 230601, China}
\email{zhouh16@mail.ustc.edu.cn}


\begin{document}

\begin{abstract}
The ribosome flow model (RFM) is a nonlinear compartmental model for unidirectional transport along an mRNA chain.  In the classical RFM, exclusion is represented by the linear free-space factor $1-x$.  We study a nonlinear steric-hindrance RFM in which this factor is replaced by $1-g(x)$, where $g(0)=0$, $g(1)=1$, and $g'>0$.  As a preliminary step, we record the dynamical properties needed for the gain-of-entrainment analysis: under constant positive rates there is a unique globally asymptotically stable equilibrium, while positive jointly periodic rates admit a unique globally attracting periodic solution.  The analysis is then built around an exact weighted edge-dissipation identity for the gain of entrainment (GOE), comparing the mean periodic production rate with the steady production rate corresponding to the averaged controls.  This identity yields a global no-GOE criterion: if $[1-g(x)]^{-1/2}$ is convex, equivalently
$$
3[g'(x)]^2+2[1-g(x)]g''(x)\geq0,
$$
then $\GOE\leq0$ for every dimension, period, forcing amplitude, and collection of positive periodic rates.  In particular, the quadratic hindrance $1-x^2$ admits no positive GOE in any dimension; moreover equality occurs only when all rates share one common positive periodic multiplier, so every genuinely nonuniform modulation strictly lowers the mean production rate.  In one dimension we obtain a sharp classification: global no-GOE for all periodic controls is equivalent to convexity of $x/[1-g(x)]$, while negative curvature at one point permits an explicit positive-GOE construction.  In higher dimensions the second variation is governed by a symmetric Jacobi matrix whose diagonal is determined by the scalar curvature and whose off-diagonal terms encode transport coupling.  We derive an exact Riccati criterion for its positivity, a dimension-uniform barrier condition for heterogeneous chains, and a closed-form critical dimension for homogeneous chains.  An explicit admissible nonlinear hindrance is exhibited for which one- and two-site homogeneous chains admit no positive second-order GOE, whereas a three-site chain has a realizable positive GOE.  Thus nonlinear steric hindrance introduces a genuine dimension-induced gain mechanism absent from the classical linear RFM.  For $g(x)=x$, the identity reduces exactly to the weighted edge-dissipation identity used by Massas and Margaliot \cite{MassasMargaliot2026Cost} to prove that periodic modulation cannot improve the average production rate of the classical RFM; the nonlinear setting reveals the curvature mechanism that determines whether this rigidity persists or fails.
\end{abstract}

\maketitle

\section{Introduction}

Protein translation is one of the central transport processes in gene expression. Ribosomes attach to a messenger RNA (mRNA), move codon by codon in a preferred direction, synthesize a polypeptide chain, and detach after completing translation. The local motion is influenced by initiation factors, codon-dependent elongation, the availability of cognate tRNAs, the finite size of ribosomes, and the structural state of the mRNA. When a ribosome slows down, trailing ribosomes may accumulate behind it, leading to crowding, collisions, and quality-control responses; see, for example, \cite{SchullerGreen2018,FilbeckEtAl2022}. From an engineering viewpoint, the same phenomena are relevant to codon optimization, heterologous expression, and synthetic control of protein production.

A classical stochastic description of ribosome traffic is the totally asymmetric simple exclusion process (TASEP); see \cite{MacDonaldEtAl1968,MacDonaldGibbs1969,BlytheEvans2007,KriecherbauerKrug2010}. Its deterministic mean-field counterpart is the ribosome flow model (RFM), in which $x_i(t)\in[0,1]$ denotes the normalized occupancy of site $i$ and the flow from site $i$ to site $i+1$ is proportional to
$$
\lambda_i(t)x_i(t)(1-x_{i+1}(t)).
$$
The factor $1-x_{i+1}$ expresses the available free space in the downstream site. The RFM has proved mathematically tractable while retaining several important qualitative features of translation. Its stability, entrainment, sensitivity, spectral structure, optimization, site-size effects, extended-particle effects, and networks with shared resources have been studied in \cite{ReuveniEtAl2011,MargaliotTuller2012,MargaliotTuller2013,MargaliotSontagTuller2014,PokerEtAl2014,RavehEtAl2016,ZaraiEtAl2017,BarShalomEtAl2020,GruneEtAl2021,JainEtAl2022,KatzEtAl2022}. Recent work also investigates feedback-induced oscillations and resource-allocation principles in RFM-based translation models \cite{EhrmanEtAl2025,MassasEtAl2026Turnpike}.

From the viewpoint of dynamical systems, translation is also a natural testing ground for questions about the interaction between temporal forcing and transport efficiency. Periodic regulation may coordinate protein production with the cell cycle or with environmental rhythms even when it does not maximize the time-averaged throughput. Thus entrainment and performance are logically distinct properties: the first asks whether the state synchronizes to periodic coefficients, whereas the second asks whether this synchronization increases or decreases the mean terminal flux. The gain-of-entrainment problem isolates precisely this second question.

The standard vacancy factor $1-x$ should be viewed as a soft mean-field closure rather than as a fundamental microscopic law. In actual translation, the effective obstruction experienced by a ribosome may depend on several unresolved microscopic features, including the finite ribosomal footprint, correlations among neighboring occupancies, local crowding and queue formation, ribosome--ribosome interactions, and downstream mRNA structures. Some of these effects are modeled explicitly in richer frameworks: the RFM with extended objects (RFMEO) accounts for the fact that one ribosome covers several codons \cite{ZaraiEtAl2017}; deterministic excluded-flow models with local interactions incorporate nearest-neighbor attraction and repulsion \cite{ZaraiKolomeiskyMargaliot2017}; and mechanistic studies of translation through structured mRNA regions emphasize the role of downstream secondary structure in modulating translocation \cite{XieChen2017}. Ribosome-profiling analyses and related models also indicate that ribosomal interference and crowding can have a substantial effect on local elongation dynamics \cite{DaoDucSong2018}. These observations support the biological interpretation of a nonlinear occupancy-response law.

Motivated by this, we consider the generalized flow law
$$
\lambda_i(t)x_i(t)[1-g(x_{i+1}(t))],
$$
where
$$
g(0)=0,\qquad g(1)=1,\qquad g'(x)>0\quad(0<x<1).
$$
Here $1-g(x)$ should be understood as an \emph{effective nonlinear availability law}: it preserves the nearest-neighbor transport architecture of the classical RFM while allowing unresolved crowding and interaction effects to modify the dependence on downstream occupancy. The quadratic choice $g(x)=x^2$ serves as the simplest nonlinear prototype. It preserves the biologically natural endpoint conditions of full availability at zero occupancy and complete blockage at saturation, while allowing the hindrance strength to vary nonlinearly at intermediate densities. We do not regard this quadratic form as a universal empirical law; rather, it is a minimal nonlinear test case that makes the role of curvature mathematically transparent.

Figure~\ref{fig:nonlinear-hindrance-schematic} summarizes the biological interpretation underlying the generalized model. In panel (A), ribosomes move along the coding region of the mRNA in the $5'\hbox{-to-}3'$ direction; the labels $5'$ and $3'$ indicate the molecular orientation of the mRNA strand. The panel also indicates several mechanisms that are represented explicitly in more detailed descriptions of translation, including the multi-codon ribosomal footprint, local ribosome--ribosome interactions and queues, downstream mRNA secondary structure, and high-density crowding. Panel (B) shows how these effects are compressed here into a one-dimensional nearest-neighbor transport law with effective availability $1-g(x)$. Panel (C) contrasts the classical linear law with nonlinear prototypes. The schematic was created specifically for the present paper and is not adapted from an existing published figure.

\begin{figure}[t]
\centering
\includegraphics[width=0.98\textwidth]{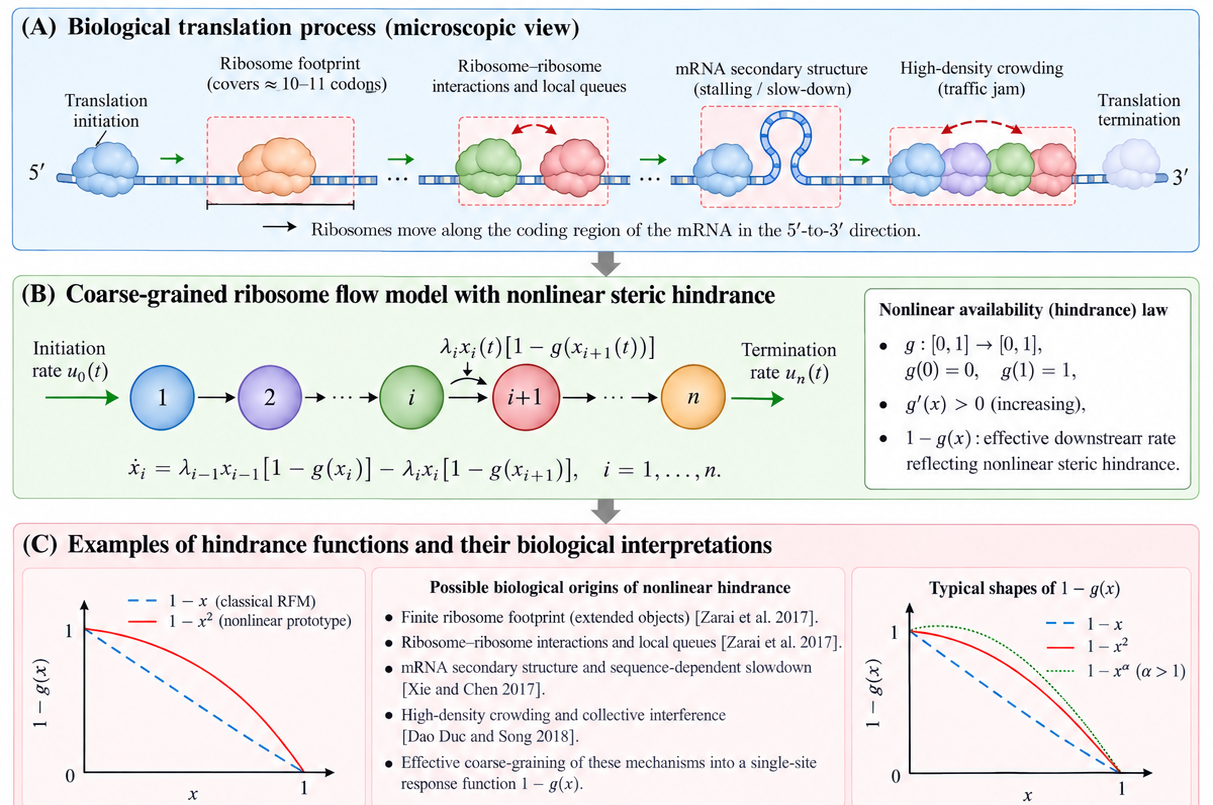}
\caption{From biological translation to a coarse-grained ribosome flow model with nonlinear steric hindrance. In panel (A), ribosomes move along the coding region of the mRNA in the $5'\hbox{-to-}3'$ direction. The panel also illustrates finite ribosome footprint, local ribosome--ribosome interactions and queues, secondary-structure-induced slowdown, and high-density crowding. Panel (B) shows the generalized nearest-neighbor RFM with effective availability $1-g(x_{i+1})$, and panel (C) compares representative hindrance profiles. The figure is an original schematic prepared for this paper and is not adapted from previously published artwork.}
\label{fig:nonlinear-hindrance-schematic}
\end{figure}

There is also a dynamical-systems reason for studying this generalization. The Jacobian of the RFM is tridiagonal and cooperative, and tridiagonal competitive/cooperative systems possess a strong variation-diminishing structure going back to Smillie and Smith \cite{Smillie1984,Smith1991}; related Lyapunov and total-positivity viewpoints were developed in \cite{FiedlerGedeon1999,MargaliotSontag2019}. More generally, the theory of monotone and strongly order-preserving systems provides a robust framework for convergence and recurrent dynamics \cite{Hirsch1982,SmithThieme1991,SmithBook}. The nonlinear law $1-g(x)$ preserves this cooperative nearest-neighbor geometry, so it remains accessible to the same qualitative machinery while moving beyond the bilinear structure of the classical model.

Our main question concerns periodic forcing. Cellular translation is not isolated from time-dependent regulation. Circadian rhythms, the cell-cycle program, periodic resource availability, and engineered oscillatory circuits can modulate initiation or elongation rates. If all transition rates are positive and jointly $T$-periodic, the classical RFM entrains to a unique $T$-periodic orbit \cite{MargaliotSontagTuller2014}. Once entrainment holds, a natural performance question arises: can a periodic modulation increase the \emph{average} protein production rate relative to the constant system obtained by replacing every rate by its temporal average? The difference between these two outputs is the gain of entrainment (GOE).

For the classical RFM, Ofir, Kriecherbauer, Gr\"une, and Margaliot \cite{OfirEtAl2023} established no-GOE results under several structural conditions. Katz, Kriecherbauer, Gr\"une, and Margaliot \cite{KatzEtAl2024} subsequently showed, in a broader weakly contractive bilinear setting, that the first-order GOE always vanishes, so gain is necessarily a second- or higher-order effect. Massas and Margaliot \cite{MassasMargaliot2026Cost} resolve the global problem for the classical linear-hindrance RFM: for arbitrary positive jointly periodic rates, the GOE is nonpositive, and equality occurs precisely when all rates share a common positive periodic modulation. Their proof introduces a weighted edge-dissipation identity and shows that every local dissipation is nonnegative.

In a separate work, Massas and Margaliot \cite{MassasMargaliot2026Linear} study a conceptually different setting in which the state dynamics are stable and linear and all nonlinearity is confined to a static output map. There the leading nonzero GOE is a second-order effect controlled by the Hessian, hence by the curvature, of the output map; convexity or concavity on the controllable subspace determines the global sign, and the gain admits a Bregman-divergence interpretation. This distinction is essential for the present work. In the nonlinear RFM studied here, the measured output remains the terminal transport flux, whereas the nonlinear function $g$ enters directly into every internal edge law and therefore changes the vector field, the periodic orbit, the equilibrium flux balance, and the spatial coupling simultaneously. Consequently, curvature is not merely an output post-processing effect. After linearization it becomes the diagonal potential of a tridiagonal Jacobi form, while conservative nearest-neighbor transport generates the off-diagonal couplings. The competition between these two effects produces a Riccati recurrence and, in particular, a critical chain length beyond which a gain-producing spatial mode may appear. This dimension-induced mechanism has no analogue in the linear-dynamics/nonlinear-output framework of \cite{MassasMargaliot2026Linear}.

The nonlinear-hindrance problem is therefore fundamentally different from both the classical RFM and the stable-linear/nonlinear-output setting. Replacing $1-x$ by $1-g(x)$ preserves the cooperative tridiagonal geometry, but it destroys the bilinear algebra used in the classical RFM and inserts nonlinearity into the transport mechanism itself. More importantly, the global rigidity theorem of Massas and Margaliot \cite{MassasMargaliot2026Cost} is no longer universal: for general increasing $g$, positive GOE can occur. The central purpose of this paper is to identify precisely which nonlinear structure preserves the classical no-GOE mechanism, when scalar curvature alone creates gain, and when gain is created only collectively through spatial coupling.

The results developed below fit into a single structural picture. We begin by recording, in a self-contained form, the nonlinear transport dynamics used later: convergence to a unique equilibrium for constant positive rates and existence of a unique globally attracting periodic orbit under positive jointly periodic forcing. On this dynamical foundation we derive an exact weighted edge-dissipation identity for the GOE. The weights are generated by $\theta(e_i)=e_i g'(e_i)/(1-g(e_i))$, and the classical formula is recovered exactly when $g(x)=x$. The dissipation identity shows that convexity of $W(x)=[1-g(x)]^{-1/2}$ guarantees global no-GOE in every dimension and for arbitrary forcing amplitude. In particular, $g(x)=x^2$ satisfies the strict curvature condition, so the quadratic hindrance $1-x^2$ admits no positive GOE in any dimension, with equality only under a common scalar modulation of all averaged rates. In one dimension the global criterion can be sharpened: no positive GOE for all periodic controls is equivalent to convexity of $\Phi(x)=x/[1-g(x)]$, whereas $\Phi''<0$ at a point permits an explicit gain-producing periodic control with the same mean rates.

In higher dimensions a different feature appears. The small-amplitude GOE is governed by a symmetric Jacobi matrix whose diagonal term $\chi(x)=1+xg'(x)/(1-g(x))+xg''(x)/(2g'(x))$ carries the scalar curvature information, while its off-diagonal entries encode nearest-neighbor transport coupling. The resulting quadratic form admits an exact Riccati characterization of positive definiteness, including a dimension-uniform criterion for heterogeneous chains and a closed-form critical dimension for homogeneous chains. This analysis also yields an explicit admissible nonlinear hindrance for which the one- and two-site homogeneous systems have no positive second-order GOE, while a three-site chain possesses a realizable positive-GOE mode. Thus the same nonlinear transport law may remain rigid at low dimension and acquire gain only through spatial coupling as the chain length increases.

The relationship with the classical no-GOE theorem of Massas and Margaliot \cite{MassasMargaliot2026Cost} is therefore precise. Our dissipation formula contains the classical formula as the special case $g(x)=x$, and the global convexity theorem recovers its no-GOE conclusion. The nonlinear remainder, however, introduces curvature. Depending on that curvature, the classical rigidity may persist globally, fail already in dimension one, or fail only after nearest-neighbor coupling produces a negative Jacobi mode. Thus the present results are not merely a replacement of $x$ by $g(x)$ in the classical proof; they identify the structural boundary between \emph{cost of entrainment} and genuine positive gain.

The paper is organized as follows. Section~\ref{sec:model} introduces the nonlinear RFM and proves the autonomous and periodically forced dynamical results. Section~\ref{sec:dissipation} derives the exact dissipation identity and the global curvature criterion, with the quadratic model as a principal corollary. Section~\ref{sec:scalar} gives the sharp one-dimensional classification. Section~\ref{sec:jacobi} develops the second-variation and Jacobi theory. Section~\ref{sec:riccati} gives the Riccati characterization, uniform heterogeneous criterion, critical-dimension formula, and dimension-induced counterexample. Section~\ref{sec:discussion} discusses the relation to recent GOE theory and several open problems.

\section{The nonlinear RFM and its dynamical foundation}\label{sec:model}

Let
$$
h(x):=1-g(x),\qquad x\in[0,1],
$$
where throughout this section
\begin{equation}
g\in C^1([0,1]),\qquad g(0)=0,\qquad g(1)=1,\qquad g'(x)>0\quad(0<x<1).
\label{eq:2.1}
\end{equation}
Then $h(0)=1$, $h(1)=0$, and $h$ is strictly decreasing on $[0,1]$.

For $n\ge1$, set
$$
x_0(t)\equiv1,\qquad x_{n+1}(t)\equiv0,
$$
and consider
\begin{equation}
\dot x_i=u_{i-1}(t)x_{i-1}h(x_i)-u_i(t)x_i h(x_{i+1}),\qquad i=1,\dots,n.
\label{eq:nlrfm}
\end{equation}
The instantaneous edge flows are
\begin{equation}
q_i(t):=u_i(t)x_i(t)h(x_{i+1}(t)),\qquad i=0,\dots,n,
\label{eq:edgeflow}
\end{equation}
so that
$$
\dot x_i=q_{i-1}-q_i.
$$
In the translation interpretation, $q_n(t)=u_n(t)x_n(t)$ is the instantaneous protein production rate.  Whenever the coefficients are $T$-periodic, we use the notation
$$
\av{f}:=\frac1T\int_0^T f(t)\,dt,
\qquad
\lambda_i:=\av{u_i}=\frac1T\int_0^T u_i(t)\,dt,
\qquad i=0,\dots,n,
$$
so that $\lambda_i$ denotes the time-averaged rate on edge $i$.

We first collect the elementary invariance and persistence properties needed later.

\begin{lemma}[Forward invariance and instantaneous interior entry]\label{lem:invariance}
Assume that every $u_i(t)$ is continuous and strictly positive.  Then $[0,1]^n$ is forward invariant for \eqref{eq:nlrfm}.  Moreover, every solution with initial condition in $[0,1]^n$ satisfies
$$
x(t)\in(0,1)^n\qquad\text{for every }t>0.
$$
\end{lemma}

\begin{proof}
If $x_i=0$, then
$$
\dot x_i=u_{i-1}(t)x_{i-1}h(0)=u_{i-1}(t)x_{i-1}\ge0,
$$
with the convention $x_0=1$.  If $x_i=1$, then
$$
\dot x_i=-u_i(t)h(x_{i+1})\le0
$$
for $i<n$, while $\dot x_n=-u_n(t)<0$.  Hence the vector field points inward or tangentially on every face of the cube, and $[0,1]^n$ is forward invariant.

We now prove strict interior entry.  Since $\dot x_1=u_0(t)>0$ whenever $x_1=0$, uniqueness implies $x_1(t)>0$ for every $t>0$.  Suppose inductively that $x_{i-1}(t)>0$ for every $t>0$.  If $x_i$ vanished at some positive time, then at its first positive-time zero one would have
$$
\dot x_i=u_{i-1}(t)x_{i-1}(t)>0,
$$
contradicting first contact with the lower boundary.  Thus all components are positive for positive time.

The upper boundary follows backwards.  Since $\dot x_n=-u_n(t)<0$ whenever $x_n=1$, we have $x_n(t)<1$ for $t>0$.  If $x_{i+1}(t)<1$ for all $t>0$ and $x_i$ reached $1$ at a first positive time, then
$$
\dot x_i=-u_i(t)h(x_{i+1}(t))<0,
$$
a contradiction.  Induction from $i=n-1$ down to $1$ completes the proof.
\end{proof}

\subsection{The autonomous system}

Assume first that
$$
u_i(t)\equiv\lambda_i>0,\qquad i=0,\dots,n.
$$
The autonomous system is strongly cooperative and tridiagonal in the interior.  Indeed,
\begin{align}
\frac{\partial f_i}{\partial x_{i-1}}&=\lambda_{i-1}h(x_i)>0,\label{eq:offlower}\\
\frac{\partial f_i}{\partial x_{i+1}}&=\lambda_i x_i g'(x_{i+1})>0,\label{eq:offupper}
\end{align}
for the relevant indices, and all other off-diagonal derivatives vanish.  This places the model in the classical Smillie--Smith framework \cite{Smillie1984,Smith1991,MargaliotSontag2019}.

The corresponding autonomous existence, uniqueness, and global-attractivity result for this nonlinear-hindrance model was also obtained in \cite{Wu2026}. We include a self-contained proof because the structural estimates will be used later in the GOE analysis.

\begin{theorem}[Unique globally attracting equilibrium]\label{thm:autonomous}
For every $\lambda_i>0$, the autonomous nonlinear RFM has a unique equilibrium
$$
e=(e_1,\dots,e_n)\in(0,1)^n.
$$
For every $a\in[0,1]^n$, the corresponding solution satisfies
$$
\lim_{t\to\infty}x(t,a)=e.
$$
\end{theorem}

\begin{proof}
By Lemma~\ref{lem:invariance}, every orbit is bounded and enters $(0,1)^n$ immediately.  The system is strongly cooperative and tridiagonal there by \eqref{eq:offlower}--\eqref{eq:offupper}.  Smillie's convergence theorem \cite{Smillie1984}, in the $C^1$ form provided by Smith \cite{Smith1991}, implies that every orbit converges to an equilibrium.  In particular, an equilibrium exists.

It remains to prove uniqueness.  At an equilibrium, conservation of flow gives a constant $R>0$ such that
\begin{equation}
\lambda_0h(e_1)=\lambda_1e_1h(e_2)=\cdots=\lambda_{n-1}e_{n-1}h(e_n)=\lambda_ne_n=R.
\label{eq:equilibriumflow}
\end{equation}
Suppose that two equilibria $e$ and $\tilde e$ correspond to currents $R$ and $\widetilde R$, respectively, and assume $R<\widetilde R$.  From the last identity in \eqref{eq:equilibriumflow},
$$
e_n<\tilde e_n.
$$
Working backwards, if $e_{i+1}<\tilde e_{i+1}$ then $h(e_{i+1})>h(\tilde e_{i+1})$, and hence
$$
e_i=\frac{R}{\lambda_i h(e_{i+1})}
<\frac{\widetilde R}{\lambda_i h(\tilde e_{i+1})}=\tilde e_i.
$$
Thus $e_1<\tilde e_1$.  But the first identity in \eqref{eq:equilibriumflow} gives
$$
R=\lambda_0h(e_1)>\lambda_0h(\tilde e_1)=\widetilde R,
$$
contradicting $R<\widetilde R$.  The case $R>\widetilde R$ is symmetric, so the equilibrium is unique.  Since every trajectory converges to an equilibrium, every trajectory converges to $e$.
\end{proof}

\begin{remark}
The proof separates two mechanisms.  Tridiagonal cooperativity gives convergence of every bounded orbit, while the flux recursion \eqref{eq:equilibriumflow} gives uniqueness.  The latter argument only uses strict monotonicity of $g$ and therefore persists for a broad class of nonlinear exclusion laws.
\end{remark}

\subsection{Uniform persistence under time-dependent forcing}

We now assume that the rates are positive and uniformly bounded:
\begin{equation}
0<m\le u_i(t)\le M<\infty,\qquad i=0,\dots,n,
\label{eq:ratebounds}
\end{equation}
for all $t\ge0$.  No periodicity is needed for the next lemma.

\begin{lemma}[Uniform persistence away from the boundary]\label{lem:persistence}
For every $\tau>0$, there exists $\delta=\delta(\tau)\in(0,1/2)$ such that every solution of \eqref{eq:nlrfm} with initial value in $[0,1]^n$ satisfies
\begin{equation}
\delta\le x_i(t)\le1-\delta,
\qquad i=1,\dots,n,
\qquad t\ge\tau.
\end{equation}
\end{lemma}

\begin{proof}
We verify lower and upper boundary repulsion and then use a finite induction along the chain.

For the lower boundary, choose $\Delta>0$ small.  If $x_1\le\Delta$, then
$$
\dot x_1
=u_0h(x_1)-u_1x_1h(x_2)
\ge m h(\Delta)-M\Delta.
$$
Since $h(0)=1$, the right-hand side is positive for all sufficiently small $\Delta$.  For $i\ge2$, if $x_i\le\Delta$ and $x_{i-1}\ge\eta>0$, then
$$
\dot x_i
\ge m\eta h(\Delta)-M\Delta,
$$
which is again positive for sufficiently small $\Delta$, now depending on $\eta$.  Starting with $x_1$ and proceeding inductively from left to right, one obtains, on consecutive subintervals of length $\tau/n$, positive lower bounds for $x_1$, then for $x_2$, and so on.  A first-contact argument shows that once the corresponding lower threshold is reached it cannot be crossed downwards.  Hence there exists $\delta_->0$ such that
$$
x_i(t)\ge\delta_-,\qquad t\ge\tau.
$$
This is the standard boundary-repulsion induction used for periodic RFM systems; see \cite{MargaliotSontagTuller2014}.

For completeness, the upper boundary can be treated without invoking a symmetry of the equation.  Define
$$
y_k:=1-x_{n+1-k},\qquad k=1,\dots,n.
$$
If $y_1\le\Delta$, then $x_n\ge1-\Delta$, and
$$
\dot y_1=-\dot x_n
=u_nx_n-u_{n-1}x_{n-1}h(x_n)
\ge m(1-\Delta)-Mh(1-\Delta),
$$
which is positive for sufficiently small $\Delta$ because $h(1)=0$.  If $k\ge2$, $y_k\le\Delta$, and $y_{k-1}\ge\eta$, then with $j=n+1-k$ we have $x_j\ge1-\Delta$ and $x_{j+1}\le1-\eta$, so
$$
\dot y_k=-\dot x_j
=u_jx_jh(x_{j+1})-u_{j-1}x_{j-1}h(x_j)
\ge m(1-\Delta)h(1-\eta)-Mh(1-\Delta).
$$
This is positive for sufficiently small $\Delta$.  The same induction now yields $y_k(t)\ge\delta_+>0$ for $t\ge\tau$, i.e.
$$
x_i(t)\le1-\delta_+.
$$
Taking $\delta=\min\{\delta_-,\delta_+\}$ proves the claim.
\end{proof}

The next result is the key strengthening of mere nonexpansiveness.  On an interior compact set the linearized transport is irreducible, and the leaks at the entrance and exit force a strict loss of $\ell^1$ mass over every positive time interval.

\begin{lemma}[Finite-time strict contraction on interior strips]\label{lem:strictcontraction}
Fix $\delta\in(0,1/2)$ and $\tau>0$.  Let
$$
K_\delta=[\delta,1-\delta]^n.
$$
There exists $q=q(\delta,\tau)\in(0,1)$ with the following property.  If two solutions $x(t)$ and $y(t)$ of \eqref{eq:nlrfm} remain in $K_\delta$ for $t\in[s,s+\tau]$, then
\begin{equation}
\|x(s+\tau)-y(s+\tau)\|_1
\le q\|x(s)-y(s)\|_1.
\end{equation}
\end{lemma}

\begin{proof}
Set $z=x-y$.  By the integral mean-value formula,
$$
\dot z=A(t)z,
\qquad
A(t)=\int_0^1J\bigl(t,y(t)+\sigma z(t)\bigr)\,d\sigma,
$$
where $J=D_xf$ is the Jacobian of \eqref{eq:nlrfm}.  Since $K_\delta$ is convex, all intermediate points lie in $K_\delta$.

The matrix $A(t)$ is Metzler and tridiagonal.  Its nonzero off-diagonal entries are bounded below by a constant $a=a(\delta)>0$.  Indeed, on $K_\delta$,
$$
u_{i-1}h(x_i)\ge m h(1-\delta)>0,
$$
and
$$
u_i x_i g'(x_{i+1})\ge m\delta\min_{[\delta,1-\delta]}g'>0.
$$
Also, all entries of $A(t)$ are uniformly bounded in absolute value by a constant $B>0$.

A direct computation of column sums gives
$$
\one^TA(t)=(-d_1(t),0,\dots,0,-d_n(t)),
$$
where
$$
d_1(t)\ge d:=m\min_{[\delta,1-\delta]}g'>0,
\qquad
d_n(t)\ge m.
$$
After replacing $d$ by $\min\{d,m\}$ we may assume both endpoint leaks are at least $d$.

Let $\Psi(t,s)$ be the transition matrix of $\dot z=A(t)z$.  Since $A$ is Metzler, $\Psi(t,s)\ge0$.  For the sum of the entries in the $j$th column,
$$
\sigma_j(t):=\one^T\Psi(t,s)e_j,
$$
we obtain
\begin{equation}
\dot\sigma_j(t)
=\one^TA(t)\Psi(t,s)e_j
\le-d\Psi_{1j}(t,s)\le0.
\end{equation}
Thus $\|\Psi(t,s)\|_1\le1$.

We now quantify strictness.  Since $A(t)+BI$ is nonnegative and every nearest-neighbor off-diagonal entry is at least $a$, the Peano--Baker expansion gives, for $t>s$,
$$
\Psi_{1j}(t,s)
\ge e^{-B(t-s)}\frac{a^{j-1}(t-s)^{j-1}}{(j-1)!}.
$$
Consequently, for $t-s\in[\tau/2,\tau]$,
$$
\Psi_{1j}(t,s)\ge\kappa,
$$
where
$$
\kappa:=e^{-B\tau}
\min_{1\le j\le n}
\frac{a^{j-1}(\tau/2)^{j-1}}{(j-1)!}>0.
$$
Integrating (2.9) over $[s+\tau/2,s+\tau]$ yields
$$
\sigma_j(s+\tau)
\le1-\frac{d\kappa\tau}{2}.
$$
Replacing the decrement by its minimum with $1/2$ if necessary, there is a number $q\in(0,1)$, independent of $j$, such that
$$
\max_j\sigma_j(s+\tau)=\|\Psi(s+\tau,s)\|_1\le q.
$$
Therefore
$$
\|z(s+\tau)\|_1\le q\|z(s)\|_1,
$$
which is (2.8).
\end{proof}

\subsection{Periodic forcing}

Assume now that each $u_i$ is continuous, positive, and $T$-periodic, and that \eqref{eq:ratebounds} holds.

The existence, uniqueness, and global attractivity of the periodic solution for the same nonlinear-hindrance class were also established in \cite{Wu2026}. The proof below is retained in order to make explicit the persistence and finite-time contraction estimates needed subsequently.

\begin{theorem}[Unique globally attracting periodic solution]\label{thm:periodic}
System \eqref{eq:nlrfm} admits a unique $T$-periodic solution
$$
\gamma(t)\in(0,1)^n.
$$
For every $a\in[0,1]^n$,
\begin{equation}
\lim_{t\to\infty}\|x(t,0,a)-\gamma(t)\|_1=0.
\label{eq:2.10}
\end{equation}
In fact, convergence is exponential after an arbitrary positive transient.
\end{theorem}

\begin{proof}
Let $P:[0,1]^n\to[0,1]^n$ be the time-$T$ map, $P(a)=x(T,0,a)$.  Lemma~\ref{lem:invariance} implies that $P$ is well defined, and standard continuous dependence gives continuity.  Brouwer's fixed-point theorem yields $a_*$ with $P(a_*)=a_*$.  The corresponding solution is $T$-periodic.

Suppose $\gamma$ and $\widetilde\gamma$ are two $T$-periodic solutions.  By compactness of their images and Lemma~\ref{lem:invariance}, there exists $\delta>0$ such that both trajectories and every line segment joining them lie in $K_\delta$.  Applying Lemma~\ref{lem:strictcontraction} with $\tau=T$ gives
$$
\|\gamma(T)-\widetilde\gamma(T)\|_1
\le q\|\gamma(0)-\widetilde\gamma(0)\|_1
$$
for some $q<1$.  Periodicity makes the left- and right-hand norms equal, so the norm must vanish.  Hence the periodic solution is unique.

Finally, fix an arbitrary solution $x(t)$.  By Lemma~\ref{lem:persistence}, after any prescribed positive transient both $x(t)$ and $\gamma(t)$ remain in a common strip $K_\delta$.  Applying Lemma~\ref{lem:strictcontraction} successively over intervals of length $T$ yields geometric decay at the discrete times $t_0+kT$.  The nonexpansive estimate $\|\Psi(t,s)\|_1\le1$ between these times gives \eqref{eq:2.10} for all $t$ and the asserted exponential convergence after the transient.
\end{proof}

\begin{corollary}[Common scalar modulation]\label{cor:timescale}
Suppose
$$
u_i(t)=m(t)\lambda_i,
\qquad i=0,\dots,n,
$$
where $m$ is positive and $T$-periodic and the $\lambda_i$ are positive constants.  Then the unique periodic solution is the constant equilibrium of the autonomous system with rates $\lambda_i$.
\end{corollary}

\begin{proof}
With the new time variable
$$
s(t)=\int_0^tm(r)\,dr,
$$
we have $ds/dt=m(t)>0$, and \eqref{eq:nlrfm} becomes the autonomous nonlinear RFM with constant rates $\lambda_i$.  By Theorem~\ref{thm:autonomous}, its unique bounded entire recurrent state is the equilibrium.  Theorem~\ref{thm:periodic} then identifies this equilibrium as the unique periodic solution in the original time variable.
\end{proof}

\begin{remark}
Theorem~\ref{thm:periodic} is consistent with the classical periodic tridiagonal theory of Smith \cite{Smith1991}, which shows that bounded solutions of periodic cooperative tridiagonal systems are asymptotic to periodic solutions.  The additional transport structure of \eqref{eq:nlrfm} gives uniqueness: the raw $\ell^1$ logarithmic norm is nonpositive because all internal Jacobian column sums vanish, and irreducible transport carries mass to the endpoint leaks, producing strict contraction over every positive finite time interval.
\end{remark}

\section{Exact dissipation and a global no-GOE theorem}\label{sec:dissipation}

We now assume that the rates $u_i$ are positive and $T$-periodic.  Recall that $\lambda_i=\av{u_i}$ denotes the averaged rate, and let $e=(e_1,\dots,e_n)$ be the unique equilibrium of the autonomous system with rates $\lambda_i$.  Set
$$
e_0:=1,\qquad e_{n+1}:=0.
$$
Define the equilibrium edge factors
\begin{equation}
c_i:=e_i h(e_{i+1}),\qquad i=0,\dots,n.
\label{eq:ci}
\end{equation}
The equilibrium flow balance gives
\begin{equation}
R_C:=\lambda_nc_n=\lambda_i c_i,\qquad i=0,\dots,n.
\label{eq:RC}
\end{equation}
Let $\gamma(t)$ be the unique periodic solution of Theorem~\ref{thm:periodic}.  Along this orbit define
$$
q_i(t):=u_i(t)\gamma_i(t)h(\gamma_{i+1}(t)),
$$
with $\gamma_0\equiv1$ and $\gamma_{n+1}\equiv0$.  Averaging
$$
\dot\gamma_i=q_{i-1}-q_i
$$
over one period yields
\begin{equation}
\av{q_0}=\cdots=\av{q_n}=:R_P.
\label{eq:RPbalance}
\end{equation}
The gain of entrainment is
\begin{equation}
\GOE(u):=R_P-R_C.
\label{eq:GOEdef}
\end{equation}

\subsection{The correct nonlinear weights}

Set
$$
z_i(t):=\gamma_i(t)-e_i,
\qquad z_0=z_{n+1}=0.
$$
The first variation of the $i$th edge factor $x_i h(x_{i+1})$ at $(e_i,e_{i+1})$ is
\begin{equation}
\eta_i:=h(e_{i+1})z_i-e_i g'(e_{i+1})z_{i+1}.
\label{eq:eta}
\end{equation}
Define
\begin{equation}
\theta(x):=\frac{xg'(x)}{h(x)}
=\frac{xg'(x)}{1-g(x)},\qquad x\in(0,1),
\label{eq:theta}
\end{equation}
and recursively choose positive weights
\begin{equation}
p_i=p_{i-1}\theta(e_i),\qquad i=1,\dots,n,
\label{eq:weights}
\end{equation}
with $p_0>0$ normalized by
$$
\sum_{i=0}^np_i=1.
$$
Finally, define the nonlinear Taylor remainder of the edge factor by
\begin{equation}
\rho_i:=\gamma_i h(\gamma_{i+1})-c_i-\eta_i.
\label{eq:rho}
\end{equation}

\begin{lemma}[Telescoping identity]\label{lem:telescoping}
For every $t$,
\begin{equation}
\sum_{i=0}^n\frac{p_i}{c_i}\eta_i(t)=0.
\label{eq:telescoping}
\end{equation}
\end{lemma}

\begin{proof}
For $i=0,\dots,n-1$, \eqref{eq:ci}, \eqref{eq:eta}, and \eqref{eq:weights} give
$$
\frac{p_i}{c_i}\eta_i
=\frac{p_i}{e_i}z_i
-\frac{p_i g'(e_{i+1})}{h(e_{i+1})}z_{i+1}
=\frac{p_i}{e_i}z_i
-\frac{p_{i+1}}{e_{i+1}}z_{i+1}.
$$
For the last edge, $c_n=e_n$ and $\eta_n=z_n$, so
$$
\frac{p_n}{c_n}\eta_n=\frac{p_n}{e_n}z_n.
$$
The sum telescopes, and the boundary term at $i=0$ vanishes because $z_0=0$.
\end{proof}

The next theorem is the central identity of the paper.

\begin{theorem}[Exact nonlinear edge-dissipation identity]\label{thm:dissipation}
For $i=0,\dots,n$, define
\begin{equation}
d_i
:=\frac{\eta_i^2}{c_i}
-\left(1-\frac{\eta_i}{c_i}\right)\rho_i.
\label{eq:di}
\end{equation}
Then
\begin{equation}
{
\GOE(u)
=-\sum_{i=0}^np_i\av{u_i d_i}.
}
\label{eq:dissidentity}
\end{equation}
\end{theorem}

\begin{proof}
Introduce the weighted quadratic energy
\begin{equation}
V(t):=\frac12\sum_{j=1}^n\frac{p_j}{e_j}z_j^2(t).
\label{eq:energy}
\end{equation}
Since $\dot z_j=q_{j-1}-q_j$, summation by parts and the proof of Lemma~\ref{lem:telescoping} yield
\begin{align*}
\dot V
&=\sum_{j=1}^n\frac{p_j}{e_j}z_j(q_{j-1}-q_j)\\
&=-\sum_{i=0}^n\frac{p_i}{c_i}q_i\eta_i.
\end{align*}
As $V$ is $T$-periodic,
\begin{equation}
\sum_{i=0}^n\frac{p_i}{c_i}\av{q_i\eta_i}=0.
\label{eq:energycancel}
\end{equation}
Now
$$
q_i=u_i(c_i+\eta_i+\rho_i),
$$
so \eqref{eq:energycancel} becomes
\begin{equation}
\sum_{i=0}^np_i\av{u_i\eta_i}
=-\sum_{i=0}^np_i\av{u_i\frac{\eta_i^2}{c_i}}
-\sum_{i=0}^np_i\av{u_i\frac{\eta_i\rho_i}{c_i}}.
\label{eq:firstsum}
\end{equation}
On the other hand, using \eqref{eq:RPbalance}, \eqref{eq:RC}, and $\sum p_i=1$,
\begin{align*}
\GOE(u)
&=\sum_{i=0}^np_i\bigl(\av{q_i}-\lambda_i c_i\bigr)\\
&=\sum_{i=0}^np_i\av{u_i(\eta_i+\rho_i)}.
\end{align*}
Eliminating the first sum by \eqref{eq:firstsum} gives
$$
\GOE(u)
=-\sum_{i=0}^np_i\av{u_i\left[
\frac{\eta_i^2}{c_i}
-\left(1-\frac{\eta_i}{c_i}\right)\rho_i
\right]},
$$
which is \eqref{eq:dissidentity}.
\end{proof}

\begin{remark}[Exact reduction to the classical edge-dissipation identity]\label{rem:classicalreduction}
Let $g(x)=x$.  Then $h(x)=1-x$ and
$$
\theta(e_i)=\frac{e_i}{1-e_i},
$$
so our recursion
$$
p_i=p_{i-1}\theta(e_i)
$$
is exactly the classical weight recursion
$$
p_i=p_{i-1}\frac{e_i}{1-e_i}
$$
used in \cite{MassasMargaliot2026Cost}, with the same normalization $\sum_{i=0}^n p_i=1$.
If $\bar\gamma_i:=\gamma_i-e_i$, then $z_i=\bar\gamma_i$ and
$$
\eta_i=(1-e_{i+1})\bar\gamma_i-e_i\bar\gamma_{i+1}.
$$
Thus $\eta_i$ is exactly the centered edge variable denoted by $h_i$ in
\cite{MassasMargaliot2026Cost}.  Moreover,
$$
\rho_i
=\gamma_i(1-\gamma_{i+1})-e_i(1-e_{i+1})-\eta_i
=-\bar\gamma_i\bar\gamma_{i+1}.
$$
Hence our edge dissipation \eqref{eq:di} becomes
$$
d_i
=\frac{\eta_i^2}{c_i}
+\left(1-\frac{\eta_i}{c_i}\right)
\bar\gamma_i\bar\gamma_{i+1},
$$
which agrees term by term, including the sign of the mixed product, with the
classical edge dissipation in \cite{MassasMargaliot2026Cost}.  Therefore
$$
\GOE(u)=-\sum_{i=0}^n p_i\av{u_i d_i}
$$
reduces exactly to the dissipation representation of Massas and Margaliot \cite{MassasMargaliot2026Cost}, not merely to an
inequality with the same sign.  The nonlinear theory below differs precisely
through the Taylor remainder $\rho_i$, which is no longer the single product
$-z_i z_{i+1}$ and whose effect is governed by curvature.
\end{remark}

\subsection{A curvature criterion for nonnegative edge dissipation}

For $i<n$, write
$$
a=e_i,\qquad b=e_{i+1},\qquad x=\gamma_i,\qquad y=\gamma_{i+1}.
$$
Define normalized variables
\begin{equation}
r:=\frac{x-a}{a},
\qquad
s:=\frac{g'(b)}{h(b)}(y-b),
\label{eq:rs}
\end{equation}
and
\begin{equation}
A_b(s):=\frac{h(y)}{h(b)}.
\label{eq:Ab}
\end{equation}
These definitions also apply to the entrance edge $i=0$, where $a=x=1$ and hence $r=0$.

\begin{lemma}[Normalized edge dissipation]\label{lem:normalizedd}
For every $i=0,\dots,n-1$,
\begin{equation}
\frac{d_i}{c_i}
=1-A_b(s)(1+r)(1-r+s),
\label{eq:dnormalized1}
\end{equation}
and equivalently
\begin{equation}
{
\frac{d_i}{c_i}
=A_b(s)\left(r-\frac{s}{2}\right)^2
+1-A_b(s)\left(1+\frac{s}{2}\right)^2.
}
\label{eq:dnormalized2}
\end{equation}
For the exit edge,
\begin{equation}
d_n=\frac{z_n^2}{e_n}\ge0.
\label{eq:lastd}
\end{equation}
\end{lemma}

\begin{proof}
By construction,
$$
\frac{\eta_i}{c_i}=r-s,
$$
and
$$
\frac{\gamma_i h(\gamma_{i+1})}{c_i}=(1+r)A_b(s).
$$
Hence
$$
\frac{\rho_i}{c_i}
=(1+r)A_b(s)-1-r+s.
$$
Substitution into \eqref{eq:di} gives \eqref{eq:dnormalized1}.  The identity
$$
(1+r)(1-r+s)
=\left(1+\frac{s}{2}\right)^2
-\left(r-\frac{s}{2}\right)^2
$$
gives \eqref{eq:dnormalized2}.  For $i=n$, $h(e_{n+1})=h(0)=1$, $z_{n+1}=0$, and the edge map is linear in $x_n$, so $\rho_n=0$, $\eta_n=z_n$, and $c_n=e_n$.
\end{proof}

Define
\begin{equation}
W(x):=h(x)^{-1/2}=[1-g(x)]^{-1/2}.
\label{eq:W}
\end{equation}

\begin{theorem}[Global no-GOE under hindrance curvature]\label{thm:globalnogoe}
Assume $g\in C^2([0,1])$ and \eqref{eq:2.1}.  If $W$ is convex on $(0,1)$, then for every $n$, every $T>0$, and every family of positive continuous $T$-periodic rates,
\begin{equation}
{R_P\le R_C.}
\label{eq:globalnogoe}
\end{equation}
If $W$ is strictly convex, then equality holds if and only if there exists a positive $T$-periodic function $m$ with $\av{m}=1$ such that
\begin{equation}
u_i(t)=m(t)\lambda_i,
\qquad i=0,\dots,n.
\label{eq:commonmod}
\end{equation}
\end{theorem}

\begin{proof}
Fix an edge $i<n$ and use the notation of Lemma~\ref{lem:normalizedd}.  Suppose first that $s>-2$.  Since
$$
W'(b)=\frac{g'(b)}{2h(b)^{3/2}},
$$
we have
$$
W(b)+W'(b)(y-b)
=W(b)\left(1+\frac{s}{2}\right)>0.
$$
Convexity of $W$ gives
$$
W(y)\ge W(b)\left(1+\frac{s}{2}\right).
$$
After squaring and using $A_b(s)=W(b)^2/W(y)^2$,
\begin{equation}
A_b(s)\left(1+\frac{s}{2}\right)^2\le1.
\label{eq:tangentineq}
\end{equation}
Thus both terms in \eqref{eq:dnormalized2} are nonnegative, so $d_i\ge0$.

If $s\le-2$, then $r>-1$ because $x=a(1+r)>0$, and therefore
$$
1-r+s<0.
$$
Since $A_b(s)>0$ and $1+r>0$, \eqref{eq:dnormalized1} gives $d_i/c_i>1$.  The last edge is nonnegative by \eqref{eq:lastd}.  Hence every $d_i$ is nonnegative, and Theorem~\ref{thm:dissipation} yields $R_P-R_C\le0$.

Assume now that $W$ is strictly convex and $R_P=R_C$.  Since every $p_i>0$, every $u_i(t)>0$, and every $d_i(t)\ge0$, the exact identity \eqref{eq:dissidentity} implies
$$
d_i(t)\equiv0
$$
for all $i$.  For $i<n$, the case $s\le-2$ is impossible.  Equality in \eqref{eq:tangentineq} and strict convexity force $y=b$, hence $s=0$; then the first term in \eqref{eq:dnormalized2} forces $r=0$.  Starting from the entrance edge and proceeding along the chain yields
$$
\gamma_i(t)\equiv e_i,
\qquad i=1,\dots,n.
$$
Thus the periodic orbit is constant.  The flow equations then imply
$$
u_i(t)c_i=Q(t)
$$
for a common instantaneous current $Q(t)$.  Since $\lambda_i c_i=R_C$,
$$
u_i(t)=\frac{Q(t)}{R_C}\lambda_i.
$$
Setting $m(t)=Q(t)/R_C$ gives \eqref{eq:commonmod}, and averaging yields $\av{m}=1$.

Conversely, if \eqref{eq:commonmod} holds, Corollary~\ref{cor:timescale} shows that the unique periodic solution is the constant equilibrium $e$, so $R_P=R_C$.
\end{proof}

A convenient pointwise version of the curvature hypothesis follows immediately.

\begin{corollary}[Differential curvature condition]\label{cor:Wcondition}
Under the assumptions of Theorem~\ref{thm:globalnogoe}, it is sufficient that
\begin{equation}
{
3[g'(x)]^2+2[1-g(x)]g''(x)\ge0,
\qquad 0<x<1.
}
\label{eq:curvaturecondition}
\end{equation}
Indeed,
\begin{equation}
W''(x)=
\frac{3[g'(x)]^2+2[1-g(x)]g''(x)}{4[1-g(x)]^{5/2}}.
\label{eq:Wsecond}
\end{equation}
\end{corollary}

\subsection{The quadratic hindrance $1-x^2$}

\begin{corollary}[Complete no-GOE theorem for $1-x^2$]\label{cor:quadratic}
Let $g(x)=x^2$.  Then for every dimension $n$, every positive jointly $T$-periodic family $u_0,\dots,u_n$, and every forcing amplitude,
\begin{equation}
{R_P\le R_C.}
\label{eq:quadraticnogoe}
\end{equation}
Moreover,
$$
R_P=R_C
$$
if and only if all rates have the common-modulation form \eqref{eq:commonmod}.  Hence every genuinely nonuniform periodic forcing satisfies
$$
R_P<R_C.
$$
\end{corollary}

\begin{proof}
For $g(x)=x^2$,
$$
g'(x)=2x,\qquad g''(x)=2,
$$
and
$$
3[g'(x)]^2+2[1-g(x)]g''(x)
=12x^2+4(1-x^2)=4+8x^2>0.
$$
Thus $W(x)=(1-x^2)^{-1/2}$ is strictly convex, and the result follows from Theorem~\ref{thm:globalnogoe}.
\end{proof}

\subsection{Further nonlinear classes covered by the global criterion}

The curvature condition is not specific to the quadratic law.  It produces broad families of nonlinear exclusion profiles for which the classical production-cost phenomenon remains globally rigid.

\begin{corollary}[Power-law hindrance]\label{cor:power}
Let
$$
g(x)=x^p,\qquad p\ge1.
$$
Then the nonlinear RFM with hindrance $1-x^p$ satisfies
$$
R_P\le R_C
$$
for arbitrary dimension, forcing amplitude, and positive jointly periodic rates.  If $p\ge1$, equality is possible only under the common-modulation condition \eqref{eq:commonmod}.
\end{corollary}

\begin{proof}
A direct computation gives
$$
3[g'(x)]^2+2[1-g(x)]g''(x)
=p x^{p-2}\left(2(p-1)+(p+2)x^p\right).
$$
This quantity is nonnegative on $(0,1)$ and is strictly positive there.  Hence $W=(1-x^p)^{-1/2}$ is strictly convex, and Theorem~\ref{thm:globalnogoe} applies.
\end{proof}

A second example is provided by a smooth saturating profile that was also used in numerical studies of the nonlinear model.

\begin{corollary}[Trigonometric hindrance]\label{cor:sin}
Let
$$
g(x)=\sin\frac{\pi x}{2}.
$$
Then the global no-GOE and rigidity conclusions of Theorem~\ref{thm:globalnogoe} hold.
\end{corollary}

\begin{proof}
Writing $z=\pi x/2$, we have
$$
g'(x)=\frac{\pi}{2}\cos z,\qquad
g''(x)=-\frac{\pi^2}{4}\sin z.
$$
Therefore
\begin{align*}
3[g'(x)]^2+2[1-g(x)]g''(x)
&=\frac{\pi^2}{4}\left[3\cos^2 z-2\sin z(1-\sin z)\right]\\
&=\frac{\pi^2}{4}(1-\sin z)(3+\sin z)>0
\end{align*}
for $0<x<1$.
\end{proof}

\begin{remark}[Global versus scalar curvature]\label{rem:globalvscalar}
The condition $W''\ge0$ is stronger than what is needed in dimension one.  Indeed, Proposition~\ref{prop:aone} below shows that $W''\ge0$ implies
$$
\chi(x)\ge1+\frac{\theta(x)}4>0.
$$
By \eqref{eq:phisecondchi}, this forces $\Phi''(x)>0$.  Thus the global all-dimensional criterion automatically lies inside the sharp scalar no-GOE region.  The converse fails: scalar convexity can survive even when the higher-dimensional Riccati barrier is lost.  This gap is exactly where coupling-induced and dimension-induced gain may occur.
\end{remark}

\section{A sharp one-dimensional classification}\label{sec:scalar}

The global criterion in Theorem~\ref{thm:globalnogoe} is dimension-independent and therefore necessarily stronger than what is needed in a scalar model.  In one dimension an exact transformation gives a sharp condition.

For $n=1$, the equation is
\begin{equation}
\dot x=u_0(t)h(x)-u_1(t)x.
\label{eq:scalar}
\end{equation}
Define
\begin{equation}
\Phi(x):=\frac{x}{h(x)}=\frac{x}{1-g(x)}.
\label{eq:Phi}
\end{equation}
Since
$$
\Phi'(x)=\frac{h(x)+xg'(x)}{h(x)^2}>0,
$$
$\Phi$ is strictly increasing from $(0,1)$ onto $(0,\infty)$.

\begin{theorem}[Sharp scalar GOE classification]\label{thm:scalarsharp}
Assume $g\in C^2([0,1])$ and \eqref{eq:2.1}.  The following are equivalent:
\begin{enumerate}[label=\textup{(\alph*)}]
\item for every pair of positive mean rates and every positive jointly periodic control with these means, the scalar model satisfies $R_P\le R_C$;
\item $\Phi(x)=x/[1-g(x)]$ is convex on $(0,1)$.
\end{enumerate}
If $\Phi$ is strictly convex, every nonconstant periodic orbit gives strict inequality.
\end{theorem}

\begin{proof}
Let
$$
a:=\av{u_0},\qquad b:=\av{u_1}.
$$
The constant-rate equilibrium $e$ satisfies
$$
ah(e)=be,
$$
or equivalently
\begin{equation}
\Phi(e)=\frac{a}{b}.
\label{eq:scalarbalance}
\end{equation}
Choose an antiderivative $H$ satisfying
$$
H'(x)=\frac1{h(x)}.
$$
Along the periodic solution $\gamma$,
$$
\frac{d}{dt}H(\gamma(t))
=u_0(t)-u_1(t)\Phi(\gamma(t)).
$$
Averaging over one period gives
\begin{equation}
\av{u_1\Phi(\gamma)}=a.
\label{eq:scalaravg}
\end{equation}
Define a probability measure on $[0,T]$ by
$$
d\mu(t)=\frac{u_1(t)}{bT}\,dt.
$$
Then \eqref{eq:scalaravg} and \eqref{eq:scalarbalance} imply
\begin{equation}
\int\Phi(\gamma)\,d\mu=\Phi(e).
\label{eq:phimean}
\end{equation}
If $\Phi$ is convex, its inverse $\Psi=\Phi^{-1}$ is concave.  Jensen's inequality gives
\begin{align*}
\frac{R_P}{b}
&=\int\gamma\,d\mu
=\int\Psi(\Phi(\gamma))\,d\mu\\
&\le\Psi\left(\int\Phi(\gamma)\,d\mu\right)
=\Psi(\Phi(e))=e.
\end{align*}
Thus $R_P\le be=R_C$.  Strict convexity of $\Phi$ makes $\Psi$ strictly concave, so equality forces $\Phi(\gamma)$, hence $\gamma$, to be constant $\mu$-almost everywhere and therefore everywhere.

Conversely, suppose $\Phi$ is not convex.  Since $\Phi\in C^2$, there exists $e\in(0,1)$ such that
$$
\Phi''(e)<0.
$$
Fix $b>0$ and set $a=b\Phi(e)$.  Let $\Psi=\Phi^{-1}$ and, for small $\delta$, define
\begin{equation}
q_\delta(t)=\Phi(e)+\delta\cos\frac{2\pi t}{T},
\qquad
\gamma_\delta(t)=\Psi(q_\delta(t)).
\label{eq:constructgamma}
\end{equation}
Take
$$
u_1(t)\equiv b
$$
and define
\begin{equation}
u_0(t):=\frac{\dot\gamma_\delta(t)+b\gamma_\delta(t)}{h(\gamma_\delta(t))}.
\label{eq:constructu0}
\end{equation}
For sufficiently small $|\delta|$, $u_0$ is positive.  By construction, $\gamma_\delta$ solves \eqref{eq:scalar}.  Moreover,
\begin{align*}
\av{u_0}
&=\av{\frac{\dot\gamma_\delta}{h(\gamma_\delta)}}
+b\av{\Phi(\gamma_\delta)}\\
&=b\Phi(e)=a,
\end{align*}
because the first average is the period average of $dH(\gamma_\delta)/dt$, and $\av{q_\delta}=\Phi(e)$.

Finally,
$$
\Psi''(\Phi(e))
=-\frac{\Phi''(e)}{[\Phi'(e)]^3}>0.
$$
Taylor expansion of \eqref{eq:constructgamma} yields
\begin{equation}
\av{\gamma_\delta}
=e+\frac{\delta^2}{4}\Psi''(\Phi(e))+o(\delta^2)>e
\label{eq:scalarpositive}
\end{equation}
for all sufficiently small nonzero $\delta$.  Therefore
$$
R_P=b\av{\gamma_\delta}>be=R_C,
$$
which contradicts (a).  Hence (a) implies convexity of $\Phi$.
\end{proof}

A useful expression for the scalar curvature is
\begin{equation}
\Phi''(x)
=\frac{2g'(x)}{[1-g(x)]^2}\chi(x),
\label{eq:phisecondchi}
\end{equation}
where
\begin{equation}
{
\chi(x):=1+\frac{xg'(x)}{1-g(x)}+\frac{xg''(x)}{2g'(x)}.
}
\label{eq:chi}
\end{equation}
Thus $\chi$ has exactly the same sign as $\Phi''$.  This quantity will become the diagonal potential of the higher-dimensional Jacobi operator.

\begin{remark}[Why the scalar and global criteria differ]
Theorem~\ref{thm:globalnogoe} requires convexity of $W=[1-g]^{-1/2}$, whereas the sharp scalar condition is convexity of $\Phi=x/[1-g]$.  The first condition is designed to make every \emph{edge dissipation} nonnegative independently of the neighboring state; it is therefore strong enough to control chains of arbitrary length.  The scalar condition only controls a single input-output balance and can hold even when the edgewise global curvature condition fails.
\end{remark}

\section{Second variation and the Jacobi operator}\label{sec:jacobi}

We now study small periodic perturbations of constant rates.  Throughout this section assume
$$
g\in C^3([0,1])
$$
in addition to \eqref{eq:2.1}.  Fix positive constants $\lambda_0,\dots,\lambda_n$ and let $e$ be their equilibrium.  Let
\begin{equation}
u_i^\varepsilon(t)=\lambda_i+\varepsilon v_i(t),
\qquad
\av{v_i}=0,
\label{eq:controlperturb}
\end{equation}
where the $v_i$ are continuous and $T$-periodic.  For small $\varepsilon$, positivity of the controls is preserved.  The periodic solution depends smoothly on $\varepsilon$; write
\begin{equation}
\gamma_i^\varepsilon(t)
=e_i+\varepsilon e_i w_i(t)+O(\varepsilon^2).
\label{eq:stateperturb}
\end{equation}
Because the mean controls are exactly $\lambda_i$, the comparison equilibrium and $R_C$ do not depend on $\varepsilon$.

The next lemma makes the remainder estimates used in the second variation
uniform in time.  In particular, the $O(\varepsilon^3)$ term below is not a
formal Taylor symbol.

\begin{lemma}[Uniform periodic-orbit and dissipation remainders]\label{lem:uniformremainder}
Fix the positive constant control $\lambda=(\lambda_0,\dots,\lambda_n)$,
its equilibrium $e$, and a continuous $T$-periodic perturbation
$v=(v_0,\dots,v_n)$ with $\av{v_i}=0$.  Assume $g\in C^3([0,1])$.
Let
$$
u_i^\varepsilon(t)=\lambda_i+\varepsilon v_i(t).
$$
Then there exist $\varepsilon_0>0$ and $C>0$ such that, for every
$|\varepsilon|\le\varepsilon_0$:
\begin{enumerate}[label=\textup{(\roman*)}]
\item $u_i^\varepsilon(t)>0$ for all $i$ and $t$, and the corresponding
$T$-periodic solution $\gamma^\varepsilon$ is well defined;
\item if $y(t)=\partial_\varepsilon\gamma^\varepsilon(t)|_{\varepsilon=0}$,
then
\begin{equation}
\sup_{t\in[0,T]}
\left|\gamma^\varepsilon(t)-e-\varepsilon y(t)\right|
\le C\varepsilon^2;
\label{eq:uniformorbitremainder}
\end{equation}
\item writing $y_j=e_jw_j$ and $w_0=0$, for every edge $i<n$, $j=i+1$,
\begin{align}
\sup_{t\in[0,T]}
\Bigg|
\frac{d_i^\varepsilon(t)}{c_i}
-&\varepsilon^2\Bigg[
 w_i^2-\theta(e_j)w_iw_j
+\left(
\theta(e_j)^2
+\frac{e_j^2g''(e_j)}{2[1-g(e_j)]}
\right)w_j^2
\Bigg]
\Bigg|
\le C|\varepsilon|^3,
\label{eq:uniformedgeremainder}
\end{align}
and
\begin{equation}
\sup_{t\in[0,T]}
\left|
\frac{d_n^\varepsilon(t)}{c_n}-\varepsilon^2w_n^2(t)
\right|
\le C|\varepsilon|^3.
\label{eq:uniformlastremainder}
\end{equation}
Consequently,
\begin{equation}
\left|
\GOE(u^\varepsilon)
+\varepsilon^2R_C\av{w^TH_n(e)w}
\right|
\le C|\varepsilon|^3.
\label{eq:uniformGOEremainder}
\end{equation}
\end{enumerate}
The constant $C$ may depend on $g,\lambda,v,T$ and $n$, but not on
$t$ or $\varepsilon$ in the indicated range.
\end{lemma}

\begin{proof}
Let $\phi_\varepsilon(t,a)$ denote the solution at time $t$ issued from
$a$ under the control $u^\varepsilon$.  Since $e\in(0,1)^n$, choose a compact
cube $K\Subset(0,1)^n$ containing $e$ in its interior.  For sufficiently
small $|\varepsilon|$ and $a$ close to $e$, solutions on $[0,T]$ remain in
$K$.  On this compact set the vector field is $C^2$ in $(a,\varepsilon)$,
with uniformly bounded derivatives up to second order, because
$g\in C^3([0,1])$ and $v$ is bounded.

Let
$$
F(a,\varepsilon):=\phi_\varepsilon(T,a)-a.
$$
At $(e,0)$,
$$
D_aF(e,0)=e^{AT}-I,
\qquad A=D_xf(e,\lambda).
$$
By the strict finite-time contraction estimate proved above for the
autonomous flow, $A$ is Hurwitz; hence $1\notin\sigma(e^{AT})$ and
$D_aF(e,0)$ is invertible.  The finite-dimensional implicit function theorem
therefore yields a $C^2$ map $a(\varepsilon)$ such that
$F(a(\varepsilon),\varepsilon)=0$.  Thus
$$
\gamma^\varepsilon(t)=\phi_\varepsilon(t,a(\varepsilon))
$$
depends $C^2$-smoothly on $\varepsilon$, uniformly for $t\in[0,T]$.
Taylor's theorem on the compact interval $[0,T]$ gives
\eqref{eq:uniformorbitremainder}.  Differentiating the state equation at
$\varepsilon=0$ shows that $y$ solves
$$
\dot y=Ay+Gv,
\qquad y(0)=y(T),
$$
where $G=D_uf(e,\lambda)$.

For an edge $i<n$, put $j=i+1$ and use the normalized variables
$r,s$ from \eqref{eq:rs}.  Estimate \eqref{eq:uniformorbitremainder}
implies, uniformly in $t$,
$$
r=\varepsilon w_i+O(\varepsilon^2),
\qquad
s=\varepsilon\theta(e_j)w_j+O(\varepsilon^2).
$$
The exact normalized dissipation
$$
D_{e_j}(r,s):=1-A_{e_j}(s)(1+r)(1-r+s)
$$
is $C^3$ in a neighborhood of $(0,0)$.  Its constant and linear terms vanish,
and its quadratic Taylor polynomial is
$$
r^2-rs+\left(1+\frac{g''(e_j)[1-g(e_j)]}
{2[g'(e_j)]^2}\right)s^2.
$$
Taylor's theorem with remainder, on a fixed small neighborhood of the origin,
gives
$$
\left|D_{e_j}(r,s)-D^{(2)}_{e_j}(r,s)\right|
\le C(|r|+|s|)^3
$$
uniformly in $t$.  Substituting the preceding uniform expansions for $r$ and
$s$ proves \eqref{eq:uniformedgeremainder}.  The last-edge estimate follows
from the exact identity $d_n/c_n=(\gamma_n-e_n)^2/e_n^2$.

Finally $u_i^\varepsilon=\lambda_i+\varepsilon v_i$ is uniformly bounded,
$d_i^\varepsilon=O(\varepsilon^2)$ uniformly, and
$\lambda_i c_i=R_C$.  Multiplying the edge expansions by
$p_i u_i^\varepsilon$, summing over the finitely many edges, and averaging
over $[0,T]$ yields \eqref{eq:uniformGOEremainder}.  Positivity of
$u_i^\varepsilon$ for sufficiently small $|\varepsilon|$ follows from
$\min_i\lambda_i>0$ and boundedness of $v$.
\end{proof}

\subsection{Expansion of the edge dissipation}

Let $j=i+1$ and consider edge $i<n$.  In the normalized variables of \eqref{eq:rs},
$$
r=\varepsilon w_i+O(\varepsilon^2),
\qquad
s=\varepsilon\theta(e_j)w_j+O(\varepsilon^2).
$$
Since
$$
h(b+\Delta)
=h(b)-g'(b)\Delta-\frac12g''(b)\Delta^2+O(\Delta^3),
$$
we have
$$
A_b(s)
=1-s-\frac{g''(b)h(b)}{2[g'(b)]^2}s^2+O(s^3).
$$
Substituting in \eqref{eq:dnormalized1} gives the following expansion.

\begin{lemma}[Second-order edge expansion]\label{lem:edgeexpansion}
For $i=0,\dots,n-1$ and $j=i+1$,
\begin{align}
\frac{d_i}{c_i}
=&\,\varepsilon^2\Bigg[
 w_i^2-\theta(e_j)w_iw_j
+\left(
\theta(e_j)^2
+\frac{e_j^2g''(e_j)}{2[1-g(e_j)]}
\right)w_j^2
\Bigg]
+O(\varepsilon^3),
\label{eq:edgeexpansion}
\end{align}
with the convention $w_0=0$.  For the last edge,
$$
\frac{d_n}{c_n}=\varepsilon^2w_n^2+O(\varepsilon^3).
$$
\end{lemma}

\begin{proof}
From \eqref{eq:dnormalized1},
$$
\frac{d_i}{c_i}=1-A_b(s)(1+r)(1-r+s).
$$
Now
$$
(1+r)(1-r+s)=1+s-r^2+rs,
$$
and
$$
A_b(s)=1-s-\beta s^2+O(s^3),
\qquad
\beta=\frac{g''(b)h(b)}{2[g'(b)]^2}.
$$
Multiplication gives
$$
\frac{d_i}{c_i}=r^2-rs+(1+\beta)s^2+O(|r|+|s|)^3.
$$
Using $s=\varepsilon\theta(b)w_j+O(\varepsilon^2)$ and
$$
\beta\theta(b)^2
=\frac{b^2g''(b)}{2h(b)}
$$
yields \eqref{eq:edgeexpansion}.  The last edge follows from \eqref{eq:lastd}.
\end{proof}

The terms multiplying $w_j^2$ from the two adjacent edges combine into the scalar curvature quantity $\chi(e_j)$.

\begin{theorem}[Jacobi second variation]\label{thm:jacobi}
Let $p_i$ be the weights \eqref{eq:weights} associated with the equilibrium $e$.  Then
\begin{equation}
{
\GOE(u^\varepsilon)
=-\varepsilon^2R_C\av{w(t)^TH_n(e)w(t)}+\mathcal R_3(\varepsilon),
}
\label{eq:jacobisecond}
\end{equation}
where $w=(w_1,\dots,w_n)^T$ and $H_n(e)$ is the symmetric tridiagonal matrix
\begin{equation}
H_n(e)=
\begin{pmatrix}
p_1\chi_1&-p_2/2&0&\cdots&0\\
-p_2/2&p_2\chi_2&-p_3/2&\ddots&\vdots\\
0&-p_3/2&p_3\chi_3&\ddots&0\\
\vdots&\ddots&\ddots&\ddots&-p_n/2\\
0&\cdots&0&-p_n/2&p_n\chi_n
\end{pmatrix},
\label{eq:Hmatrix}
\end{equation}
with
$$
\chi_j:=\chi(e_j),
$$
and the remainder satisfies
$$
|\mathcal R_3(\varepsilon)|\le C|\varepsilon|^3
$$
for $|\varepsilon|\le\varepsilon_0$, with $C,\varepsilon_0$ as in
Lemma~\ref{lem:uniformremainder}.
\end{theorem}

\begin{proof}
By Lemma~\ref{lem:uniformremainder}, the edge expansions hold with a
remainder bounded by $C|\varepsilon|^3$ uniformly in time.  Combining these
bounds with Theorem~\ref{thm:dissipation} and $\lambda_i c_i=R_C$ yields
$$
\GOE(u^\varepsilon)
=-\varepsilon^2R_C\av{Q_e(w)}+\mathcal R_3(\varepsilon),
\qquad
|\mathcal R_3(\varepsilon)|\le C|\varepsilon|^3.
$$
where each edge $i=j-1$ contributes
$$
p_{j-1}\left[-\theta(e_j)w_{j-1}w_j
+\left(\theta(e_j)^2+\frac{e_j^2g''(e_j)}{2h(e_j)}\right)w_j^2\right],
$$
and edge $j$ contributes $p_jw_j^2$.  Since
$$
p_j=p_{j-1}\theta(e_j),
$$
the cross term is $-p_jw_{j-1}w_j$.  The diagonal coefficient becomes
\begin{align*}
p_j
&+p_{j-1}\theta(e_j)^2
+p_{j-1}\frac{e_j^2g''(e_j)}{2h(e_j)}\\
&=p_j\left[
1+\theta(e_j)+\frac{e_jg''(e_j)}{2g'(e_j)}
\right]
=p_j\chi(e_j).
\end{align*}
Therefore
$$
Q_e(w)
=\sum_{j=1}^np_j\chi_jw_j^2
-\sum_{j=2}^np_jw_{j-1}w_j
=w^TH_n(e)w,
$$
which proves \eqref{eq:jacobisecond}.
\end{proof}

The preceding theorem gives the quadratic form, but a negative direction matters for GOE only if it can actually be generated by a zero-mean periodic rate perturbation.  The path-graph control structure implies full realizability.

\begin{theorem}[Realizability of negative Jacobi modes]\label{thm:realizability}
Let $H_n(e)$ be the matrix in Theorem~\ref{thm:jacobi}.  The following are equivalent:
\begin{enumerate}[label=\textup{(\roman*)}]
\item $H_n(e)$ has a negative eigenvalue;
\item there exists a zero-mean $T$-periodic perturbation $v=(v_0,\dots,v_n)$ such that the corresponding family \eqref{eq:controlperturb} has positive second-order GOE.
\end{enumerate}
More precisely, if $\xi^TH_n(e)\xi<0$, then one can choose a periodic perturbation for which
\begin{equation}
\GOE(u^\varepsilon)
=-\frac{\varepsilon^2R_C}{2}\xi^TH_n(e)\xi+O(\varepsilon^3)>0
\label{eq:positivequadratic}
\end{equation}
for all sufficiently small nonzero $\varepsilon$.
\end{theorem}

\begin{proof}
Linearization at the equilibrium and constant rates gives
\begin{equation}
\dot y=Ay+Gv,
\label{eq:linearcontrol}
\end{equation}
where
$$
A=D_xf(e,\lambda),
\qquad
G=D_uf(e,\lambda).
$$
The $i$th rate perturbs the $i$th edge flow, so
$$
G=B\diag(c_0,\dots,c_n),
$$
where $B$ is the node-edge incidence matrix of the directed path with $n$ nodes and $n+1$ boundary/internal edges.  Hence
$$
\operatorname{rank}G=n.
$$
Choose a right inverse $G^R$.

Given $\xi\in\R^n$, set
$$
y(t)=\diag(e_1,\dots,e_n)\xi\cos(\omega t),
\qquad
\omega=\frac{2\pi k}{T},\ k\in\N.
$$
Define
\begin{equation}
v(t):=G^R\bigl(\dot y(t)-Ay(t)\bigr).
\label{eq:controlrightinverse}
\end{equation}
Then \eqref{eq:linearcontrol} holds and
$$
\av{v}=G^R(\av{\dot y}-A\av{y})=0.
$$
For the autonomous system, choose an interior strip $K_\delta$ containing $e$ in its interior and apply Lemma~\ref{lem:strictcontraction} to nearby trajectories.  Differentiating the resulting finite-time contraction estimate at the equilibrium gives $\|e^{A\tau}\|_1<1$ for some $\tau>0$, and hence the spectral bound of $A$ is strictly negative.  In particular, $1\notin\sigma(e^{AT})$.  The periodic fixed-point equation is therefore locally solvable by the implicit function theorem, and the nonlinear periodic orbit generated by $u^\varepsilon=\lambda+\varepsilon v$ has first variation $y$.  In the notation of \eqref{eq:stateperturb},
$$
w(t)=\xi\cos(\omega t).
$$
Substituting this into \eqref{eq:jacobisecond} gives
$$
\GOE(u^\varepsilon)
=-\varepsilon^2R_C\av{\cos^2(\omega t)}\xi^TH_n(e)\xi+O(\varepsilon^3),
$$
which is \eqref{eq:positivequadratic} because $\av{\cos^2}=1/2$.

If $H_n(e)\ge0$, then \eqref{eq:jacobisecond} is nonpositive at second order for every periodic first variation.  This proves the equivalence.
\end{proof}

\begin{remark}
Equation \eqref{eq:phisecondchi} shows that $\chi$ is the scalar curvature controlling the one-site problem.  The higher-dimensional Jacobi matrix adds negative nearest-neighbor couplings.  Thus a chain may acquire a negative mode even when every $\chi(e_j)$ is positive.  This is the mechanism behind dimension-induced GOE.
\end{remark}

\section{Riccati criteria and dimension-induced GOE}\label{sec:riccati}

The sign of the second variation can be characterized exactly by a scalar Riccati recurrence.  This turns a dynamical GOE question into a one-dimensional positivity problem for a Jacobi matrix.

Let
$$
P:=\diag(p_1,\dots,p_n).
$$
By congruence,
$$
H_n>0
\quad\Longleftrightarrow\quad
K_n:=P^{-1/2}H_nP^{-1/2}>0.
$$
Using $p_j/p_{j-1}=\theta(e_j)$,
\begin{equation}
K_n=
\begin{pmatrix}
\chi_1&-\sqrt{\theta_2}/2&0&\cdots&0\\
-\sqrt{\theta_2}/2&\chi_2&-\sqrt{\theta_3}/2&\ddots&\vdots\\
0&-\sqrt{\theta_3}/2&\chi_3&\ddots&0\\
\vdots&\ddots&\ddots&\ddots&-\sqrt{\theta_n}/2\\
0&\cdots&0&-\sqrt{\theta_n}/2&\chi_n
\end{pmatrix},
\label{eq:Kmatrix}
\end{equation}
where $\theta_j=\theta(e_j)$.

\begin{theorem}[Exact Riccati positivity criterion]\label{thm:riccati}
Define
\begin{equation}
r_1:=\chi(e_1),
\qquad
r_k:=\chi(e_k)-\frac{\theta(e_k)}{4r_{k-1}},
\quad k=2,\dots,n,
\label{eq:riccatirec}
\end{equation}
whenever the preceding denominators are nonzero.  Then
\begin{equation}
{
H_n(e)>0
\quad\Longleftrightarrow\quad
r_1>0,\dots,r_n>0.
}
\label{eq:riccatiequiv}
\end{equation}
Equivalently, if
$$
D_0:=1,\qquad D_1:=\chi(e_1),
$$
and
\begin{equation}
D_k=\chi(e_k)D_{k-1}-\frac{\theta(e_k)}4D_{k-2},
\qquad k\ge2,
\label{eq:detrec}
\end{equation}
then
$$
H_n(e)>0
\quad\Longleftrightarrow\quad
D_k>0\quad(k=1,\dots,n).
$$
\end{theorem}

\begin{proof}
The matrix $K_n$ is symmetric tridiagonal.  In its $LDL^T$ factorization, the successive pivots are
$$
r_1=\chi_1,
$$
and
$$
r_k=\chi_k-\frac{(\sqrt{\theta_k}/2)^2}{r_{k-1}}
=\chi_k-\frac{\theta_k}{4r_{k-1}}.
$$
A symmetric matrix is positive definite if and only if all its $LDL^T$ pivots are positive.  Since $H_n$ and $K_n$ are congruent, \eqref{eq:riccatiequiv} follows.  The determinant recurrence is the standard recurrence for the leading principal minors of a Jacobi matrix, and $r_k=D_k/D_{k-1}$ whenever all preceding minors are nonzero.
\end{proof}

\subsection{Uniformly heterogeneous chains}

Fix a compact interval
$$
K\Subset(0,1).
$$
We ask for a condition ensuring positive second variation for \emph{every} chain length and every choice $e_j\in K$.

\begin{theorem}[Dimension-uniform heterogeneous criterion]\label{thm:uniformbarrier}
The following statements are equivalent:
\begin{enumerate}[label=\textup{(\roman*)}]
\item for every $n\ge1$ and every $e_1,\dots,e_n\in K$, the Jacobi matrix $H_n(e)$ is positive definite;
\item there exists $a>0$ such that
\begin{equation}
{
\chi(x)\ge a+\frac{\theta(x)}{4a}
\qquad\text{for every }x\in K.
}
\label{eq:barrier}
\end{equation}
\end{enumerate}
\end{theorem}

\begin{proof}
Assume \eqref{eq:barrier}.  Then $r_1=\chi(e_1)>a$.  If $r_{k-1}\ge a$, the Riccati recurrence gives
$$
r_k
=\chi(e_k)-\frac{\theta(e_k)}{4r_{k-1}}
\ge\chi(e_k)-\frac{\theta(e_k)}{4a}
\ge a.
$$
Thus all pivots are positive for every length and every heterogeneous chain.

Conversely, define the worst possible pivots recursively by
\begin{equation}
a_1:=\min_{x\in K}\chi(x),
\qquad
a_{m+1}:=\min_{x\in K}\left[\chi(x)-\frac{\theta(x)}{4a_m}\right].
\label{eq:worstpivots}
\end{equation}
Because $r\mapsto\chi(x)-\theta(x)/(4r)$ is strictly increasing for $r>0$, induction shows that $a_m$ is the infimum of the $m$th pivot among all length-$m$ chains with entries in $K$.  Assumption (i) therefore implies $a_m>0$ for every $m$.

Since $a_2<a_1$ and the map
$$
\mathcal T(r):=\min_{x\in K}\left[\chi(x)-\frac{\theta(x)}{4r}\right]
$$
is nondecreasing on $(0,\infty)$, induction gives $a_{m+1}\le a_m$ for all $m$.  Hence $a_m$ has a nonnegative limit.  Since $K\Subset(0,1)$,
$$
\theta_*:=\min_{x\in K}\theta(x)>0,
$$
and $\chi$ is bounded above on $K$.  The sequence cannot tend to zero, because then the negative term $-\theta_*/(4a_m)$ in \eqref{eq:worstpivots} would force $a_{m+1}<0$.  Hence
$$
a_m\downarrow a_*>0.
$$
Continuity and compactness allow passage to the limit in \eqref{eq:worstpivots}:
$$
a_*=\min_{x\in K}\left[\chi(x)-\frac{\theta(x)}{4a_*}\right].
$$
Thus
$$
\chi(x)\ge a_*+\frac{\theta(x)}{4a_*}
$$
for all $x\in K$, proving (ii).
\end{proof}

The barrier can be written as an interval-intersection condition.  For fixed $x$, \eqref{eq:barrier} is equivalent to
$$
4a^2-4\chi(x)a+\theta(x)\le0.
$$
When $\chi(x)^2\ge\theta(x)$, define
\begin{equation}
a_\pm(x)
:=\frac{\chi(x)\pm\sqrt{\chi(x)^2-\theta(x)}}2.
\label{eq:apm}
\end{equation}

\begin{corollary}[Interval-intersection form]\label{cor:intersection}
Condition \eqref{eq:barrier} is equivalent to
\begin{equation}
{
\bigcap_{x\in K}[a_-(x),a_+(x)]\ne\varnothing,
}
\label{eq:intersection}
\end{equation}
or, equivalently,
$$
\sup_{x\in K}a_-(x)\le\inf_{x\in K}a_+(x).
$$
\end{corollary}

The global condition in Section~\ref{sec:dissipation} has a transparent interpretation in this Riccati language.

\begin{proposition}[Global curvature as the barrier $a=1$]\label{prop:aone}
For every $x\in(0,1)$,
\begin{equation}
{
\chi(x)-1-\frac{\theta(x)}4
=
\frac{x}{4[1-g(x)]g'(x)}
\left(3[g'(x)]^2+2[1-g(x)]g''(x)\right).
}
\label{eq:aoneidentity}
\end{equation}
Hence the global convexity condition \eqref{eq:curvaturecondition} implies the dimension-uniform local barrier \eqref{eq:barrier} with $a=1$.
\end{proposition}

\begin{proof}
Substitute \eqref{eq:theta} and \eqref{eq:chi} into the left-hand side and collect terms over the denominator $4[1-g(x)]g'(x)$.
\end{proof}

\subsection{Homogeneous chains and a critical dimension}

Suppose
$$
e_1=\cdots=e_n=e.
$$
This can always be realized by a suitable constant-rate vector.  Put
$$
\theta:=\theta(e),\qquad \chi:=\chi(e).
$$
Then $K_n$ is the Toeplitz Jacobi matrix with diagonal $\chi$ and off-diagonal $-\sqrt\theta/2$.

\begin{proposition}[Exact homogeneous spectrum]\label{prop:toeplitz}
The eigenvalues of $K_n$ are
\begin{equation}
\mu_k
=\chi-\sqrt\theta\cos\frac{k\pi}{n+1},
\qquad k=1,\dots,n.
\label{eq:toeplitzeigs}
\end{equation}
Hence
\begin{equation}
\mu_{\min}(K_n)
=\chi-\sqrt\theta\cos\frac{\pi}{n+1}.
\label{eq:minspec}
\end{equation}
A positive second-order GOE direction exists if and only if
\begin{equation}
{
\chi(e)<\sqrt{\theta(e)}\cos\frac{\pi}{n+1}.
}
\label{eq:homogthreshold}
\end{equation}
\end{proposition}

\begin{proof}
The eigenvectors of the Toeplitz Jacobi matrix have components
$$
z_j^{(k)}=\sin\frac{jk\pi}{n+1},
$$
and direct substitution gives \eqref{eq:toeplitzeigs}.  The smallest eigenvalue corresponds to $k=1$.  The last assertion follows from Theorem~\ref{thm:realizability}.
\end{proof}

Define
$$
\varrho(e):=\frac{\chi(e)}{\sqrt{\theta(e)}}.
$$
If $0\le\varrho(e)<1$, then gain does not necessarily occur in low dimension, but it must occur once the chain is sufficiently long.

\begin{corollary}[Critical dimension]\label{cor:criticaldimension}
If $0\le\varrho(e)<1$, the smallest homogeneous dimension with a negative Jacobi mode is
\begin{equation}
{
n_{\rm crit}
=\left\lfloor\frac{\pi}{\arccos\varrho(e)}\right\rfloor.
}
\label{eq:ncrit}
\end{equation}
If $\chi(e)<0$, gain is already possible for $n=1$.  If $\chi(e)\ge\sqrt{\theta(e)}$, no finite homogeneous chain has a negative second-variation mode at the equilibrium $e$.
\end{corollary}

\begin{proof}
For $0\le\varrho<1$, condition \eqref{eq:homogthreshold} is
$$
\cos\frac{\pi}{n+1}>\varrho.
$$
Since cosine is decreasing on $(0,\pi)$, this is equivalent to
$$
\frac{\pi}{n+1}<\arccos\varrho,
$$
or $n+1>\pi/\arccos\varrho$.  The smallest integer $n$ satisfying the strict inequality is exactly \eqref{eq:ncrit}.  The remaining statements follow directly from \eqref{eq:minspec}.
\end{proof}

\subsection{An explicit dimension-induced GOE example}

We now show that the higher-dimensional mechanism is genuine and not an artifact of the formal second variation.

\begin{example}[Gain created by increasing the chain length]\label{ex:dimensiongain}
Consider
\begin{equation}
g_*(x)
:=x+\frac32x(1-x)(1-2x)
=3x^3-\frac92x^2+\frac52x.
\label{eq:gstar}
\end{equation}
Then
$$
g_*(0)=0,\qquad g_*(1)=1,
$$
and
$$
g_*'(x)=9\left(x-\frac12\right)^2+\frac14>0,
$$
so $g_*$ is admissible.

Choose the homogeneous equilibrium
$$
e=\frac13.
$$
A direct calculation gives
$$
g_*(e)=\frac49,
\qquad
1-g_*(e)=\frac59,
\qquad
g_*'(e)=\frac12,
\qquad
g_*''(e)=-3.
$$
Consequently
\begin{equation}
\theta(e)=\frac3{10},
\qquad
\chi(e)=\frac3{10}.
\label{eq:exampletheta}
\end{equation}
For $n=1$, $H_1$ is positive because $\chi(e)>0$.  For $n=2$,
$$
\mu_{\min}(K_2)
=\frac3{10}-\frac12\sqrt{\frac3{10}}>0,
$$
since $9/100>3/40$.  Thus no positive second-order direction exists for the one- or two-site homogeneous chain at this equilibrium.

For $n=3$, however,
$$
\mu_{\min}(K_3)
=\frac3{10}-\frac1{\sqrt2}\sqrt{\frac3{10}}
=\frac3{10}-\sqrt{\frac3{20}}<0,
$$
because $9/100<3/20$.  Hence Theorem~\ref{thm:realizability} produces positive GOE for a sufficiently small periodic perturbation of the three-site chain.

To make the mean control explicit, normalize $R_C=1$.  The equilibrium $e_1=e_2=e_3=1/3$ is generated by
\begin{equation}
\lambda_0=\frac95,
\qquad
\lambda_1=\lambda_2=\frac{27}{5},
\qquad
\lambda_3=3,
\label{eq:examplemeans}
\end{equation}
because
$$
\lambda_0[1-g_*(e)]
=\lambda_1e[1-g_*(e)]
=\lambda_2e[1-g_*(e)]
=\lambda_3e=1.
$$
The negative eigenvector of $K_3$ corresponding to its smallest eigenvalue is proportional to
$$
z=(1,\sqrt2,1)^T.
$$
Since $H_3$ and $K_3$ are congruent, an associated state direction is
$$
\xi=P^{-1/2}z.
$$
Taking
$$
y(t)=\diag(e,e,e)\xi\cos\frac{2\pi t}{T}
$$
and using the right-inverse construction \eqref{eq:controlrightinverse} gives a zero-mean periodic rate perturbation $v$.  For sufficiently small $\varepsilon$, the rates
$$
u_i^\varepsilon(t)=\lambda_i+\varepsilon v_i(t)
$$
remain positive, preserve the means \eqref{eq:examplemeans}, and satisfy
$$
\GOE(u^\varepsilon)>0.
$$
Thus increasing the chain from two sites to three creates a realizable gain mode.  We call this phenomenon \emph{dimension-induced gain of entrainment}.
\end{example}

\begin{remark}
The example highlights a distinction between local scalar curvature and spatial coupling.  At $e=1/3$, the scalar curvature quantity $\chi(e)$ is positive.  Nevertheless, in a long enough chain the cumulative negative off-diagonal coupling in the Jacobi operator overcomes the positive diagonal potential.  The resulting gain is therefore created by the transport architecture, not by a negative scalar curvature at an individual site.
\end{remark}

\section{Discussion and open problems}\label{sec:discussion}

The classical RFM has an unusually rigid GOE structure. Massas and Margaliot \cite{MassasMargaliot2026Cost} show that periodic modulation of the linear-hindrance RFM can never improve average production and that equality occurs only under a common scalar modulation of all rates.  The exact dissipation identity in Theorem~\ref{thm:dissipation} shows that this rigidity is not an isolated algebraic coincidence: it persists for a substantial nonlinear class, but the correct structural condition is a curvature condition on the transformed free-space profile
$$
W(x)=[1-g(x)]^{-1/2}.
$$
When $W$ is convex, every local edge dissipation is nonnegative.  In this regime periodic coordination has an unavoidable production cost, exactly as in the classical RFM.

The quadratic hindrance $1-x^2$ belongs to this global rigidity class.  This is significant for two reasons.  First, the model is genuinely nonlinear in the blocking variable and is not bilinear in the state.  Second, the conclusion is completely uniform: dimension, period, amplitude, and spatial placement of the forcing do not matter.  Thus the no-GOE phenomenon is robust under at least one natural nonlinear deformation of the exclusion law, and in fact under the larger family $1-x^p$, $p\ge1$.

On the other hand, monotonicity of $g$ is not enough.  The scalar classification in Theorem~\ref{thm:scalarsharp} identifies $\Phi=x/(1-g)$ as the sharp one-dimensional object.  If $\Phi$ loses convexity, positive GOE can be created by an explicit periodic orbit/control construction.  This has a close conceptual relation to recent curvature-based GOE results for linear systems with nonlinear outputs \cite{MassasMargaliot2026Linear}, but here the nonlinearity lies inside the transport law itself.  The higher-dimensional analysis then reveals a second mechanism that cannot be seen in a scalar Jensen argument: nearest-neighbor coupling creates a Jacobi operator, and its negative modes may appear only after the chain becomes sufficiently long.

There is also a useful comparison with the nonlinear-cascade analysis of Massas and Margaliot \cite{MassasMargaliot2026Cascade}, where GOE decomposes into weighted local Jensen gaps.  The RFM is not a feedforward cascade of static nonlinearities: conservation couples adjacent sites dynamically in both the state equation and the periodic balance.  This produces the Riccati/Jacobi structure in Sections~\ref{sec:jacobi}--\ref{sec:riccati}.  The resulting dimension threshold is therefore a specifically transport-induced effect.

It is useful to compare the preceding results with several recent developments on GOE for the RFM and related systems.  The development of GOE results for the RFM may be viewed in three stages.  The first stage was structural.  Ofir et al.~\cite{OfirEtAl2023} proved absence of gain for several families of periodically forced classical RFMs by deriving identities for centered periodic moments.  These results demonstrated substantial rigidity, but the closure of the moment identities depended strongly on the linear exclusion law.  This is precisely the point at which a direct extension to $1-g(x)$ becomes ineffective: nonlinear exclusion produces higher-order remainders that do not close under the same finite moment scheme.

The second stage was local and control-theoretic.  Katz et al.~\cite{KatzEtAl2024} considered weakly contractive bilinear systems and proved that the first-order GOE vanishes.  In this sense the performance effect of a small periodic perturbation is intrinsically second order.  Theorem~\ref{thm:jacobi} is consistent with that principle but identifies a much more specific second-order object for nonlinear RFM transport: a Jacobi form with a scalar curvature potential and nearest-neighbor couplings.  The realizability argument in Theorem~\ref{thm:realizability} further shows that negative modes of this form correspond to actual zero-mean rate perturbations rather than merely formal variations of the state.

The third stage is the recent global result of Massas and Margaliot \cite{MassasMargaliot2026Cost}.  For the classical RFM they prove $R_P\le R_C$ for arbitrary positive periodic rates and characterize equality by a common scalar modulation.  Our Theorem~\ref{thm:dissipation} is deliberately formulated so that, at $g(x)=x$, both the weights and edge dissipation reduce to that classical identity.  The main new issue is what happens away from $g(x)=x$.  The answer is not a universal extension of the classical theorem.  The nonlinear Taylor remainder carries curvature, and that curvature has two qualitatively different consequences: it yields the global rigidity class in Theorem~\ref{thm:globalnogoe}, but it can also produce positive gain, first through loss of scalar convexity and then, more subtly, through a negative spatial Jacobi mode.

Massas and Margaliot \cite{MassasMargaliot2026Linear} show that, for stable linear state dynamics with a nonlinear output, curvature controls the leading GOE and yields a Bregman-divergence interpretation.  In a separate study of nonlinear feedforward cascades, Massas and Margaliot \cite{MassasMargaliot2026Cascade} decompose the gain into downstream-weighted local Jensen gaps.  Our scalar theorem shares the Jensen/curvature mechanism, whereas the multi-site RFM introduces a conservative transport coupling that has no scalar analogue.  The resulting Riccati recurrence and critical dimension are therefore not inherited from either framework.

The biological interpretation of these distinctions should also be kept separate from mathematical optimality.  A negative GOE does not mean that periodic regulation is biologically useless: temporal coordination may be required to match protein supply to periodic demand.  Rather, the no-GOE theorem quantifies a throughput cost of such coordination under a fixed mean-rate budget.  Conversely, when a nonlinear hindrance admits positive GOE, the gain is a consequence of how time-dependent occupancy reshapes effective free space.  The dimension-induced example shows that this effect may be collective: no individual site needs to possess a locally gain-producing scalar curvature, yet the chain can support a gain-producing spatial mode.

The results above also narrow the remaining questions and provide concrete structures with which to approach them.  The global theorem shows that convexity of $W=[1-g]^{-1/2}$ is a robust sufficient condition for no-GOE at arbitrary amplitude, whereas the scalar theorem and the Riccati criterion identify explicit mechanisms by which this rigidity can fail.  A natural next problem is therefore to determine whether arbitrary-amplitude no-GOE in dimensions $n\ge2$ admits a sharp characterization that interpolates between the exact scalar criterion and the local Jacobi--Riccati test; the exact edge-dissipation identity obtained here provides a plausible starting point for such a nonlinear reduction.  When positive GOE is possible, the negative Jacobi mode and its realizability already identify the optimal infinitesimal direction, suggesting a systematic route toward finite-amplitude optimization under amplitude, energy, or rate-budget constraints, in connection with the weakly contractive control framework of \cite{KatzEtAl2024} and recent resource-allocation questions \cite{MassasEtAl2026Turnpike}.  It is also natural to replace common-period forcing by non-common periodic, almost periodic, or more general recurrent rates; the persistence and tridiagonal structure established here indicate that the curvature and Riccati mechanisms should remain meaningful after an appropriate long-time averaging formulation.  Finally, for networks of RFMs coupled through finite ribosome pools \cite{RavehEtAl2016,JainEtAl2022,KatzEtAl2022}, the present edge-dissipation identity suggests a way to separate local nonlinear blockage from topology-dependent coupling, raising the possibility of network analogues of the dimension-induced gain found here.  Thus the present theory does more than leave these questions open: it isolates the quantities--edge dissipation, scalar curvature, Jacobi coupling, and Riccati pivots--that a sharper global or network theory will have to control.

In summary, nonlinear steric hindrance changes the GOE problem qualitatively.  The classical global no-GOE theorem is recovered when $g(x)=x$ and persists under the global curvature condition \eqref{eq:curvaturecondition}, including the important model $g(x)=x^2$.  Outside this rigidity region, however, positive GOE is possible, and higher-dimensional transport can create gain even when the local scalar curvature is favorable.  The exact dissipation identity, the scalar convexity classification, and the Jacobi--Riccati reduction provide a unified framework for distinguishing these regimes.

\section*{Acknowledgements}
The author would like to thank Professor Yi Wang  and Linzhen Wu  in University of Science and Technology of China for their helpful discussions and valuable suggestions on this work. This paper is supported by the National Natural Science Foundation of China (No.~12517710, No.~12331006), the Strategic Priority Research Program of CAS (No.~XDB0900100), and the National Key R\&D Program of China (No.~2024YFA1013603, 2024YFA1013600).

\end{document}